\documentclass[reqno]{amsart}

\usepackage{amsmath,amsfonts,amssymb,amscd,verbatim,delarray,fullpage}

\usepackage{stmaryrd}
\usepackage{amsthm}
\usepackage{url}
\usepackage[polutonikogreek,english]{babel}

\DeclareSymbolFont{bbold}{U}{bbold}{m}{n}
\DeclareSymbolFontAlphabet{\mathbbold}{bbold}

\DeclareMathOperator{\spl}{split}
\DeclareMathOperator{\nsp}{non-split}
\DeclareMathOperator{\aut}{Aut}

\DeclareMathOperator{\tors}{tors}
\DeclareMathOperator{\GL}{GL}
\DeclareMathOperator{\SL}{SL}
\DeclareMathOperator{\PGL}{PGL}

\DeclareMathOperator{\tr}{tr}

\chardef\bslash=`\\ 

\begin{document}


\newtheorem{Theorem}{Theorem}[section]

\newtheorem{cor}[Theorem]{Corollary}

\newtheorem{Conjecture}[Theorem]{Conjecture}

\newtheorem{exercise}[Theorem]{Exercise}
\newtheorem{Question}[Theorem]{Question}
\newtheorem{lemma}[Theorem]{Lemma}
\newtheorem{property}[Theorem]{Property}
\newtheorem{proposition}[Theorem]{Proposition}
\newtheorem{ax}[Theorem]{Axiom}
\newtheorem{claim}[Theorem]{Claim}

\newtheorem{nTheorem}{Surjectivity Theorem}

\theoremstyle{definition}
\newtheorem{Definition}[Theorem]{Definition}
\newtheorem{problem}[Theorem]{Problem}
\newtheorem{question}[Theorem]{Question}
\newtheorem{Example}[Theorem]{Example}

\newtheorem{remark}[Theorem]{Remark}
\newtheorem{diagram}{Diagram}
\newtheorem{Remark}[Theorem]{Remark}
\newcommand{\diagref}[1]{diagram~\ref{#1}}
\newcommand{\thmref}[1]{Theorem~\ref{#1}}
\newcommand{\secref}[1]{Section~\ref{#1}}
\newcommand{\subsecref}[1]{Subsection~\ref{#1}}
\newcommand{\lemref}[1]{Lemma~\ref{#1}}
\newcommand{\corref}[1]{Corollary~\ref{#1}}
\newcommand{\exampref}[1]{Example~\ref{#1}}
\newcommand{\remarkref}[1]{Remark~\ref{#1}}
\newcommand{\corlref}[1]{Corollary~\ref{#1}}
\newcommand{\claimref}[1]{Claim~\ref{#1}}
\newcommand{\defnref}[1]{Definition~\ref{#1}}
\newcommand{\propref}[1]{Proposition~\ref{#1}}
\newcommand{\prref}[1]{Property~\ref{#1}}
\newcommand{\itemref}[1]{(\ref{#1})}
\newcommand{\ul}[1]{\underline{#1}}


\newcommand{\CE}{\mathcal{E}}
\newcommand{\CG}{\mathcal{G}}\newcommand{\CV}{\mathcal{V}}
\newcommand{\CL}{\mathcal{L}}
\newcommand{\CM}{\mathcal{M}}
\newcommand{\A}{\mathcal{A}}
\newcommand{\CO}{\mathcal{O}}
\newcommand{\B}{\mathcal{B}}
\newcommand{\CS}{\mathcal{S}}
\newcommand{\CX}{\mathcal{X}}
\newcommand{\CY}{\mathcal{Y}}
\newcommand{\CT}{\mathcal{T}}
\newcommand{\CW}{\mathcal{W}}
\newcommand{\CJ}{\mathcal{J}}

\newcommand{\st}{\sigma}
\renewcommand{\k}{\varkappa}
\newcommand{\Frac}{\mbox{Frac}}
\newcommand{\XC}{\mathcal{X}}
\newcommand{\wt}{\widetilde}
\newcommand{\wh}{\widehat}
\newcommand{\mk}{\medskip}
\renewcommand{\sectionmark}[1]{}
\renewcommand{\Im}{\operatorname{Im}}
\renewcommand{\Re}{\operatorname{Re}}
\newcommand{\la}{\langle}
\newcommand{\ra}{\rangle}
\newcommand{\LND}{\mbox{LND}}
\newcommand{\Pic}{\mbox{Pic}}
\newcommand{\lnd}{\mbox{lnd}}
\newcommand{\GLND}{\mbox{GLND}}\newcommand{\glnd}{\mbox{glnd}}
\newcommand{\Der}{\mbox{DER}}\newcommand{\DER}{\mbox{DER}}
\renewcommand{\th}{\theta}
\newcommand{\ve}{\varepsilon}
\newcommand{\1}{^{-1}}
\newcommand{\iy}{\infty}
\newcommand{\iintl}{\iint\limits}
\newcommand{\capl}{\operatornamewithlimits{\bigcap}\limits}
\newcommand{\cupl}{\operatornamewithlimits{\bigcup}\limits}
\newcommand{\suml}{\sum\limits}
\newcommand{\ord}{\operatorname{ord}}
\newcommand{\gal}{\operatorname{Gal}}
\newcommand{\bk}{\bigskip}
\newcommand{\fc}{\frac}
\newcommand{\g}{\gamma}
\newcommand{\be}{\beta}
\newcommand{\dl}{\delta}
\newcommand{\Dl}{\Delta}
\newcommand{\lm}{\lambda}
\newcommand{\Lm}{\Lambda}
\newcommand{\om}{\omega}
\newcommand{\ov}{\overline}
\newcommand{\vp}{\varphi}
\newcommand{\kap}{\varkappa}

\newcommand{\Vp}{\Phi}
\newcommand{\Varphi}{\Phi}
\newcommand{\BC}{\mathbb{C}}
\newcommand{\C}{\mathbb{C}}\newcommand{\BP}{\mathbb{P}}
\newcommand{\BQ}{\mathbb {Q}}
\newcommand{\BM}{\mathbb{M}}
\newcommand{\BR}{\mathbb{R}}\newcommand{\BN}{\mathbb{N}}
\newcommand{\BZ}{\mathbb{Z}}\newcommand{\BF}{\mathbb{F}}
\newcommand{\BA}{\mathbb {A}}
\renewcommand{\Im}{\operatorname{Im}}
\newcommand{\idd}{\operatorname{id}}
\newcommand{\ep}{\epsilon}
\newcommand{\tp}{\tilde\partial}
\newcommand{\doe}{\overset{\text{def}}{=}}
\newcommand{\supp} {\operatorname{supp}}
\newcommand{\loc} {\operatorname{loc}}
\newcommand{\de}{\partial}
\newcommand{\z}{\zeta}
\renewcommand{\a}{\alpha}
\newcommand{\G}{\Gamma}
\newcommand{\der}{\mbox{DER}}

\newcommand{\Spec}{\operatorname{Spec}}
\newcommand{\Sym}{\operatorname{Sym}}
\newcommand{\Aut}{\operatorname{Aut}}

\newcommand{\Idd}{\operatorname{Id}}

\newcommand{\tG}{\widetilde G}

\newcommand{\FX}{\mathfrac {X}}
\newcommand{\FV}{\mathfrac {V}}
\newcommand{\SX}{\mathcal {X}}
\newcommand{\SV}{\mathcal {V}}
\newcommand{\SO}{\mathcal {O}}
\newcommand{\SD}{\mathcal {D}}
\newcommand{\Sr}{\rho}
\newcommand{\SR}{\mathcal {R}}
\newcommand{\cl}{\mathcal{C}}
\newcommand{\ok}{\mathcal{O}_K}
\newcommand{\ab}{\mathcal{AB}}

\setcounter{equation}{0} \setcounter{section}{0}

\newcommand{\ds}{\displaystyle}
\newcommand{\gl}{\lambda}
\newcommand{\gL}{\Lambda}
\newcommand{\gge}{\epsilon}
\newcommand{\gG}{\Gamma}
\newcommand{\ga}{\alpha}
\newcommand{\gb}{\beta}
\newcommand{\gd}{\delta}
\newcommand{\gD}{\Delta}
\newcommand{\gs}{\sigma}
\newcommand{\mbq}{\mathbb{Q}}
\newcommand{\mbr}{\mathbb{R}}
\newcommand{\mbz}{\mathbb{Z}}
\newcommand{\mbc}{\mathbb{C}}
\newcommand{\mbn}{\mathbb{N}}
\newcommand{\mbp}{\mathbb{P}}
\newcommand{\mbf}{\mathbb{F}}
\newcommand{\mbe}{\mathbb{E}}
\newcommand{\lcm}{\text{lcm}\,}
\newcommand{\mf}[1]{\mathfrak{#1}}
\newcommand{\ol}[1]{\overline{#1}}
\newcommand{\mc}[1]{\mathcal{#1}}
\newcommand{\nequiv}{\equiv\hspace{-.07in}/\;}
\newcommand{\bnequiv}{\equiv\hspace{-.13in}/\;}

\title{Elliptic curves with $2$-torsion contained in the $3$-torsion field}
\author{J. Brau and N. Jones}


\date{}

\begin{abstract}
There is a modular curve $X'(6)$ of level $6$ defined over $\mbq$ whose $\mbq$-rational points correspond to $j$-invariants of elliptic curves $E$ over $\mbq$ that satisfy $\mbq(E[2]) \subseteq \mbq(E[3])$.  
In this note we characterize the $j$-invariants of elliptic curves with this property by exhibiting an explicit model of $X'(6)$.  Our motivation is two-fold:  on the one hand, $X'(6)$ belongs to the list of modular curves which parametrize non-Serre curves (and is not well-known), and on the other hand, $X'(6)(\mbq)$ gives an infinite family of examples of elliptic curves with non-abelian ``entanglement fields,'' which is relevant to the systematic study of correction factors of various conjectural constants for elliptic curves over $\mbq$.
\end{abstract}

\maketitle

\section{Introduction} \label{introduction}

Let $K$ be a number field, let $E$ be an elliptic curve over $K$, and for any positive integer $n$, let $E[n]$ denote the $n$-torsion of $E$.  For a prime $\ell$, let $\ds E[\ell^\infty] := \bigcup_{m \geq 1} E[\ell^m]$, and furthermore put $\ds E_{\tors} := \bigcup_{n \geq 1} E[n]$.  Fixing a $\hat{\mbz}$-basis of $E_{\tors}$, for any prime $\ell$ there is an induced $\mbz_\ell$-basis of $E[\ell^\infty]$ and for any $n \geq 1$ there is an induced $\mbz/n\mbz$-basis of $E[n]$.  Consider the Galois representations
\begin{equation*} 
\begin{split}
\rho_{E,n} \colon \gal(\ol{K}/K) &\longrightarrow \aut(E[n]) \simeq \GL_2(\mbz/n\mbz) \\
\rho_{E,\ell^\infty} \colon \gal(\ol{K}/K) &\longrightarrow \aut(E[\ell^\infty]) \simeq \GL_2(\mbz_\ell) \\
\rho_E \colon \gal(\ol{K}/K) &\longrightarrow \aut(E_{\tors}) \simeq \GL_2(\hat{\mbz}),
\end{split}
\end{equation*}
each defined by letting $\gal(\ol{K}/K)$ act on the appropriate set of torsion points, viewed relative to the appropriate basis.  

A celebrated theorem of Serre \cite{serre} states that, if $E$ is an elliptic curve over a number field $K$ without complex multiplication (``non-CM''), then the Galois representation $\rho_E$ has an open image with respect to the profinite topology on $\GL_2(\hat{\mbz})$, which is to 
say that
$
[ \GL_2(\hat{\mbz}) : \rho_E( \gal(\ol{K}/K)) ] < \infty.
$
It is of interest to understand the image of $\rho_E$.  To determine $\rho_E( \gal(\ol{K}/K) )$ in practice, one begins by computing the $\ell$-adic image $\rho_{E,\ell^\infty}( \gal(\ol{K}/K))$ for each prime $\ell$.
One then has that
\[
\rho_E(\gal(\ol{K}/K)) \hookrightarrow \prod_{\ell} \rho_{E,\ell^\infty}( \gal(\ol{K}/K)) \subseteq \prod_{\ell} \GL_2(\mbz_\ell) \simeq \GL_2(\hat{\mbz}),
\]
and although the image of $\rho_E(\gal(\ol{K}/K))$ in $\ds \prod_{\ell} \rho_{E,\ell^\infty}( \gal(\ol{K}/K))$ projects onto each $\ell$-adic factor, the inclusion may nevertheless be onto a proper subgroup.  Understanding the image of $\ds \rho_E(\gal(\ol{K}/K)) \hookrightarrow  \prod_{\ell} \rho_{E,\ell^\infty}( \gal(\ol{K}/K))$ now amounts to understanding the \emph{entanglement fields}
\[
K(E[m_1]) \cap K(E[m_2]),
\]
for each pair $m_1, m_2 \in \mbn$ which are relatively prime\footnote{Here and throughout the paper, $K(E[n]) := \ol{K}^{\ker \rho_{E,n}}$ denotes the \emph{$n$-th division field} of $E$.}.  Note that any such intersection is necessarily Galois over $K$.  One of the questions which motivates this note is the following.

\begin{question} \label{entanglementquestion}
Given a number field $K$, can one classify the triples $(E, m_1, m_2)$ with $E$ an elliptic curve over $K$ and $m_1, m_2$ a pair of co-prime integers for which the entanglement field $K(E[m_1]) \cap K(E[m_2])$ is non-abelian over $K$?
\end{question}

This question is closely related to the study of correction factors of various conjectural constants for elliptic curves over $\mathbb{Q}$. In order to illustrate this point, consider the following elliptic curve analogue Artin's conjecture on primitive roots. For an elliptic curve $E$ over $\mathbb{Q}$, determine the density of primes $p$ such that $E$ has good reduction at $p$ and $\tilde{E}(\mathbb{F}_p)$ is a cyclic group, where $\tilde{E}$ denotes the mod $p$ reduction of $E$. Note that the condition of $\tilde{E}(\mathbb{F}_p)$ being cyclic is completely determined by $\rho_E(\gal(\ol{\mathbb{Q}}/\mathbb{Q}))$. Indeed,  $\tilde{E}(\mathbb{F}_p)$ is a cyclic group if and only if $p$ does not split completely in the field extension $\mathbb{Q}(E[\ell])$ for any $\ell \neq p$.

By the Chebotarev density theorem, the set of primes $p$ that do not split completely in $\mathbb{Q}(E[\ell])$ has density equal to 
$$
\delta_\ell=1-\frac{1}{[\mathbb{Q}(E[\ell]):\mathbb{Q}]}.
$$
If we assume that the various splitting conditions at each prime $\ell$ are independent, then it is reasonable to conjecture that the density of primes $p$ for which $\tilde{E}(\mathbb{F}_p)$ is cyclic is equal to $\prod_\ell \delta_\ell$. However, this assumption of independence is not correct, and this lack of independence is explained by the entanglement fields. 

Serre showed in \cite{serre2} that Hooley's method of proving Artin's conjecture on primitive roots can be adapted to prove that the density of primes $p$ for which $\tilde{E}(\mathbb{F}_p)$ is cyclic is given under GRH by the inclusion-exclusion sum
\begin{equation}\label{HooleyDensity}
\delta(E) = \sum_{n=1}^\infty \frac{\mu(n)}{[\mbq(E[n]):\mathbb{Q}]}
\end{equation}
where $\mu$ denotes the M{\"o}bius function. Taking into account entanglements between the various torsion fields implies that
$$
\delta(E) = C_E \prod_\ell \delta_\ell
$$
where $C_E$ is an \emph{entanglement correction factor}, and explicitely evaluating such densities amounts to computing the correction factors $C_E$. When all the entanglements fields of an elliptic curve over $\mathbb{Q}$ are abelian, then the image of $\ds \rho_E(\gal(\ol{K}/K)) \hookrightarrow  \prod_{\ell} \rho_{E,\ell^\infty}( \gal(\ol{K}/K))$ is cut out by characters, and the correction factor can be given as a character sum. This method has the advantage that it is well-suited to deal with many other problems of this nature where the explicit evaluation of (\ref{HooleyDensity}) becomes problematic.  Understanding which non-abelian entanglements can occur is therefore important for the systematic study of such constants.

With respect to entanglement fields, the case $K = \mbq$, although it is usually the first case considered, has a complication which doesn't arise over any other number field.  Indeed, when the base field is $\mbq$, the Kronecker-Weber theorem, together with the containment $\mbq(\zeta_n) \subseteq \mbq(E[n])$, \emph{forces} the occurrence of non-trivial entanglement fields\footnote{Here and throughout the paper, $\zeta_n$ denotes a primitive $n$-th root of unity.}.  It was observed by Serre \cite[Proposition 22]{serre} that for any elliptic curve $E$ over $\mbq$ one has
\begin{equation} \label{serreentanglement}
\mbq(\sqrt{\gD_E}) \subseteq \mbq(E[2]) \cap \mbq(\zeta_n),
\end{equation}
where $n = 4| \gD_E |$.  This containment forces $\rho_{E}(\gal(\ol{\mbq}/\mbq))$ to lie in an appropriate index two subgroup of $\GL_2(\hat{\mbz})$, so that one must have
\begin{equation} \label{serrecurvecriterion}
[ \GL_2(\hat{\mbz}) : \rho_E(\gal(\ol{\mbq}/\mbq)) ] \geq 2.
\end{equation}
Several examples are known of elliptic curves $E$ over $\mbq$ for which the entanglement \eqref{serreentanglement} is the only obstruction to surjectivity of $\rho_E$, i.e. for which equality holds in \eqref{serrecurvecriterion}.
\begin{Definition}
We call an elliptic curve $E$ defined over $\mbq$ a \textbf{Serre curve} if $[ \GL_2(\hat{\mbz}) : \rho_E(\gal(\ol{\mbq}/\mbq)) ] = 2$.
\end{Definition}
In \cite{jones} it is shown using sieve methods that, when taken by height, almost all elliptic curves $E$ over $\mbq$ are Serre curves (see also \cite{zywina}, which generalizes this to the case $K \neq \mbq$, and \cite{radhakrishnan}, which sharpens the upper bound to an asymptotic formula).  In \cite{cogrjo}, different ideas are used to deduce stronger upper bounds for the number of elliptic curves in \emph{one-parameter} families which are not Serre curves.  These results are obtained by viewing non-Serre curves as coming from rational points on modular curves.  More precisely, there is a family
$ 
\mc{X} = \{ X_1, X_2, \dots \}
$
of modular curves with the property that, for each elliptic curve $E$, one has
\begin{equation} \label{nonserrecurvesaspoints}
\text{$E$ is not a Serre curve } \; \Longleftrightarrow \; j(E) \in \bigcup_{X \in \mc{X}} j(X(\mbq)),
\end{equation}
where $j$ denotes the natural projection followed by the usual $j$-map:
\begin{equation} \label{jmap}
j : X \longrightarrow X(1) \longrightarrow \mbp^1.
\end{equation}
In \cite{cogrjo}, the authors use \eqref{nonserrecurvesaspoints} together with geometric methods to bound the number of non-Serre curves in a given one-parameter family.  This brings us to the following question, which serves as additional motivation for the present note.
\begin{question} \label{nonserrecurvesquestion}
Consider the family $\mc{X}$ occurring in \eqref{nonserrecurvesaspoints}.  What is an explicit list of the modular curves in $\mc{X}$?
\end{question}
The modular curves in $\mc{X}$ of prime level $\ell$ correspond to maximal proper subgroups of $\GL_2(\mbz/\ell\mbz)$ and have been studied extensively.  Let
\begin{equation} \label{exceptionalsetatp}
\mc{E}_\ell \subseteq \left\{ X_0(\ell), X_{\spl}^+(\ell), X_{\nsp}^+(\ell), X_{A_4}(\ell), X_{S_4}(\ell), X_{A_5}(\ell)\right\} 
\end{equation}
be the set of modular curves whose rational points correspond to $j$-invariants of elliptic curves $E$ for which $\rho_{E,\ell}$ is not surjective (each of the modular curves $ X_{A_4}(\ell)$, $X_{S_4}(\ell)$, and $X_{A_5}(\ell)$  corresponding to the exceptional groups $A_4$, $S_4$ and $A_5$ only occurs for certain primes $\ell$).  One has
\[
\bigcup_{\ell \text{ prime}} \mc{E}_\ell \subseteq \mc{X}.
\]
The family $\mc{X}$ must also contain two other modular curves $X'(4)$ and $X''(4)$ of level $4$, and another $X'(9)$ of level $9$, which have been considered in \cite{dokchitser} and \cite{elkies}, respectively.

In this note, we consider a modular curve $X'(6)$ of level $6$ which, taken together with those listed above, completes the set $\mc{X}$ of modular curves occurring in \eqref{nonserrecurvesaspoints}, answering Question \ref{nonserrecurvesquestion}.  First, we recall the general construction of modular curves associated to subgroups $H \subseteq \GL_2(\mbz/n\mbz)$ (for more details, see \cite{delignerapoport}).  Let $X(n)$ denote the complete modular curve of level $n$, which parametrizes elliptic curves together with 
chosen $\mbz/n\mbz$-bases of $E[n]$.  Let $H \subseteq \GL_2(\mbz/n\mbz)$ be a subgroup containing $-I$ for which the determinant map
\begin{equation*} 
\det \colon H \longrightarrow (\mbz/n\mbz)^\times
\end{equation*}
is surjective, and consider the quotient curve  $X_H := X(n)/H$ together with the $j$-invariant
\[
j \colon X_H \longrightarrow \mbp^1.
\]
For any $x \in \mbp^1(\mbq)$, we have that
\begin{equation} \label{modularinterpretation}
x \in j(X_H(\mbq)) \; \Longleftrightarrow \; \begin{matrix}
\exists \text{ an elliptic curve $E$ over $\mbq$ and $\exists g \in \GL_2(\mbz/n\mbz)$} \\
\text{with $j(E) = x$  and $\rho_{E,n}(\gal(\ol{\mbq}/\mbq)) \subseteq g^{-1} H g$. }
\end{matrix}
\end{equation}
Thus, to describe $X'(6)$, it suffices to describe the corresponding subgroup $H \subseteq \GL_2(\mbz/6\mbz)$.

There is exactly one index $6$ normal subgroup $\mc{N} \subseteq \GL_2(\mbz/3\mbz)$, defined by
\begin{equation} \label{defofN}
\mc{N} := \left\{ 
\begin{pmatrix}
x & -y \\
y & x
\end{pmatrix} : \; x^2 + y^2 \equiv 1 \mod 3
\right\} \sqcup 
\left\{  
\begin{pmatrix}
x & y \\
y & -x
\end{pmatrix} : \; x^2 + y^2 \equiv -1 \mod 3
\right\}.
\end{equation}
This subgroup fits into an exact sequence
\begin{equation} \label{exactsequenceforN}
1 \longrightarrow \mc{N} \longrightarrow \GL_2(\mbz/3\mbz) \longrightarrow \GL_2(\mbz/2\mbz) \longrightarrow 1,
\end{equation}
and we denote by 
\begin{equation} \label{defofpsi}
\theta \colon \GL_2(\mbz/3\mbz) \longrightarrow \GL_2(\mbz/2\mbz)
\end{equation}
the surjective map in the above sequence.  We take $H \subseteq \GL_2(\mbz/2\mbz) \times \GL_2(\mbz/3\mbz)$ to be the graph of $\theta$, viewed as a subgroup of $\GL_2(\mbz/6\mbz)$ via the Chinese Remainder Theorem.  The modular curve $X'(6)$ is then defined by 
\begin{equation} \label{defofH}
X'(6) := X_{H_6'}, \; \text{ where } \; H_6' := \{ (g_2,g_3) \in \GL_2(\mbz/2\mbz) \times \GL_2(\mbz/3\mbz) : \; g_2 = \theta(g_3) \} \subseteq \GL_2(\mbz/6\mbz).
\end{equation}
Unravelling \eqref{modularinterpretation} in this case, we find that, for every elliptic curve $E$ over $\mbq$,
\begin{equation} \label{characterizationofXprimeof6}
j(E) \in j(X'(6)(\mbq)) \; \Longleftrightarrow \; E \simeq_{\ol{\mbq}} E' \; \text{ for some } E' \text{ over } \mbq \text{ for which } \mbq(E'[2]) \subseteq \mbq(E'[3]).
\end{equation}
By considering the geometry of the natural map $X'(6) \longrightarrow X(1)$, the curve $X'(6)$ is seen to have genus zero and one cusp.  Since $\gal(\ol{\mbq}/\mbq)$ acts on the cusps, the single cusp must be defined over $\mbq$, thus endowing $X'(6)$ with a rational point.  Therefore $X'(6) \simeq_\mbq \mbp^1$.   We prove the following theorem, which gives an explicit model of $X'(6)$.
\begin{Theorem} \label{uniformizertheorem}  There exists a uniformizer
$
\ds t \colon X'(6) \longrightarrow \mbp^1
$
with the property that
\[
j = 2^{10} 3^3 t^3 (1 - 4t^3),
\]
where $j \colon X'(6) \longrightarrow X(1) \simeq \mbp^1$ is the usual $j$-map.  
\end{Theorem}
\begin{remark}
By \eqref{characterizationofXprimeof6}, Theorem \ref{uniformizertheorem} is equivalent to the following statement:  for any elliptic curve $E$ over $\mbq$, $E$ is isomorphic over $\ol{\mbq}$ to an elliptic curve $E'$ satisfying
\[
\mbq(E'[2]) \subseteq \mbq(E'[3])
\]
if and only if
$
j(E) = 2^{10} 3^3 t^3 (1 - 4t^3)
$
for some $t \in \mbq$.
\end{remark}

Furthermore, we prove the following theorem, which answers Question \ref{nonserrecurvesquestion}.  For each prime $\ell$, consider the set $\mc{G}_{\ell,\text{max}}$ of maximal proper subgroups of $\GL_2(\mbz/\ell\mbz)$, which surject via determinant onto $(\mbz/\ell\mbz)^\times$:
\[
\mc{G}_{\ell,\text{max}} := \{ H \subsetneq \GL_2(\mbz/\ell\mbz) : \; \det(H) = (\mbz/\ell\mbz)^\times \text{ and } \nexists H_1 \text{ with } H \subsetneq H_1 \subsetneq \GL_2(\mbz/\ell\mbz) \}.
\]
The group $\GL_2(\mbz/\ell\mbz)$ acts on $\mc{G}_{\ell,\text{max}}$ by conjugation, and let $\mc{R}_\ell$ be a set of representatives of $\mc{G}_{\ell,\text{max}}$ modulo this action.  By \eqref{modularinterpretation}, the collection $\mc{X}$ occurring in \eqref{nonserrecurvesaspoints} must contain as a subset
\begin{equation} \label{Ep}
\mc{E}_\ell := \{ X_H : \; H \in \mc{R}_\ell \},
\end{equation}
the set of modular curves attached to subgroups $H \in \mc{R}_\ell$ (this gives a more precise description of the set $\mc{E}_\ell$ in \eqref{exceptionalsetatp}).  Furthermore, the previously mentioned modular curves $X'(4)$, $X''(4)$, and $X'(9)$ correspond to the following subgroups.  Let $\ve : \GL_2(\mbz/2\mbz) \longrightarrow \{ \pm 1 \}$ denote the unique non-trivial character, and we will view $\det \colon \GL_2(\mbz/4\mbz) \longrightarrow (\mbz/4\mbz)^\times \simeq \{ \pm 1 \}$ as taking the values $\pm 1$.
\begin{equation} \label{X4andX9}
\begin{split}
X'(4) = X_{H_4'}, \; \text{ where } \; &H_4' := \{ g \in \GL_2(\mbz/4\mbz) : \; \det g = \ve(g \mod 2) \} \subseteq \GL_2(\mbz/4\mbz), \\
X''(4) = X_{H_4''} \; \text{ where } \; &H_4'' := 
\left\langle
\begin{pmatrix}
0 & 1 \\
3 & 0
\end{pmatrix}, 
\begin{pmatrix}
0 & 1 \\
1 & 1
\end{pmatrix}
\right\rangle \subseteq \GL_2(\mbz/4\mbz) \\
X'(9) = X_{H_9'} \; \text{ where } \; &H_9' := 
\left\langle
\begin{pmatrix}
0 & 2 \\
4 & 0
\end{pmatrix}, 
\begin{pmatrix}
4 & 1 \\
-3 & 4
\end{pmatrix},
\begin{pmatrix}
2 & 0 \\
0 & 2
\end{pmatrix},
\begin{pmatrix}
-1 & 0 \\
0 & 1
\end{pmatrix}
\right\rangle \subseteq \GL_2(\mbz/9\mbz).
\end{split}
\end{equation}
For more details on these modular curves, see \cite{dokchitser} and  \cite{elkies}.
\begin{Theorem} \label{completelisttheorem}
Let $\mc{X}$ be defined by
\[
\mc{X} = \left\{ X'(4), X''(4), X'(9), X'(6) \right\} \cup \bigcup_{\ell \text{ prime}} \mc{E}_\ell,
\]
where $X'(4)$, $X''(4)$ and $X'(9)$ are defined by \eqref{X4andX9}, $X'(6)$ is defined by \eqref{defofH}, and $\mc{E}_\ell$ is as in \eqref{Ep}.  Then, for any elliptic curve $E$ over $\mbq$,
\[
\text{$E$ is not a Serre curve } \; \Longleftrightarrow \; j(E) \in \bigcup_{X \in \mc{X}} j(X(\mbq)).
\]
\end{Theorem}

\section{Proofs}

We now prove Theorems \ref{uniformizertheorem} and \ref{completelisttheorem}.  

\noindent \emph{Proof of Theorem \ref{uniformizertheorem}.}
Consider the elliptic curve $\mbe$ over $\mbq(t)$ given by
\begin{equation} \label{defofmbe}
\mbe : \; y^2 = x^3 + 3t \left(1-4t^3 \right) x + \left(1 - 4t^3 \right) \left( \frac{1}{2} - 4t^3 \right),
\end{equation}
with discriminant and $j$-invariant $\gD_\mbe, j(\mbe) \in \mbq(t)$ given, respectively, by
\begin{equation} \label{gDandj}
\gD_{\mbe} = - 2^6 3^3 (1 - 4t^3)^2 \quad \text{ and } \quad j(\mbe) = 2^{10} 3^3 t^3 (1 - 4t^3).
\end{equation}
For every $t \in \mbq$, the specialization $\mbe_t$ is an elliptic curve over $\mbq$ whose discriminant $\gD_{\mbe_t} \in \mbq$ and $j$-invariant $j(\mbe_t) \in \mbq$ are given by evaluating \eqref{gDandj} at $t$.  We will show that, for any $t \in \mbq$, one has
\begin{equation} \label{specializedcontainment}
\mbq(\mbe_t[2]) \subseteq \mbq(\mbe_t[3]). 
\end{equation}
By \eqref{characterizationofXprimeof6} and \eqref{gDandj}, it then follows that
\[
\forall t \in \mbq, \quad 2^{10} 3^3 t^3 (1 - 4t^3) \in j(X'(6)(\mbq)).
\]
Since the natural $j$-map $j \colon X'(6) \longrightarrow \mbp^1$ and the map $t \mapsto 2^{10} 3^3 t^3 (1 - 4t^3)$ both have degree $6$, Theorem \ref{uniformizertheorem} will then follow.  To verify \eqref{specializedcontainment}, we will show that, for every $t \in \mbq$, one has
\begin{equation} \label{morespecializedcontainment}
\mbq(\mbe_t[2]) \subseteq \mbq(\zeta_3, \gD_{\mbe_t}^{1/3} ).
\end{equation}
Taken together with the classical fact that, for any elliptic curve $E$ over $\mbq$, one has $\mbq(\zeta_3, \gD_{E}^{1/3}) \subseteq \mbq(E[3])$, the containment \eqref{specializedcontainment} then follows.  Finally, \eqref{morespecializedcontainment} follows immediately from the factorization
\[
\left( x - e_1(t) \right)  \left(  x - e_2(t)  \right)  \left(  x - e_3(t)  \right)  = x^3 + 3t \left(1-4t^3 \right) x + \left(1 - 4t^3 \right) \left( \frac{1}{2} - 4t^3 \right),
\]
of the $2$-division polynomial $\ds x^3 + 3t \left(1-4t^3 \right) x + \left(1 - 4t^3 \right) \left( \frac{1}{2} - 4t^3 \right)$, where
\[
\begin{split}
e_1(t) &:= \frac{1}{6} \gD_{\mbe_t}^{1/3} + \frac{t}{18(1-4t^3)} \gD_{\mbe_t}^{2/3}, \\
e_2(t) &:= \frac{\zeta_3}{6} \gD_{\mbe_t}^{1/3} + \frac{\zeta_3^2 t}{18(1-4t^3)} \gD_{\mbe_t}^{2/3}, \; \text{ and} \\
e_3(t) &:= \frac{\zeta_3^2}{6} \gD_{\mbe_t}^{1/3} + \frac{\zeta_3 t}{18(1-4t^3)} \gD_{\mbe_t}^{2/3}.
\end{split}
\]
This finishes the proof of Theorem \ref{uniformizertheorem}.  \hfill $\Box$

We will now turn to Theorem \ref{completelisttheorem}, whose proof employs the following two group-theoretic lemmas.
\begin{lemma} \label{goursatlemma} (Goursat's Lemma)
Let $G_0$ and $G_1$ be groups and $G \subseteq G_0 \times G_1$ a subgroup satisfying
\[
\pi_i(G) = G_i \quad\quad (i \in \{ 0, 1 \} ),
\]
where $\pi_i$ denotes the canonical projection onto the $i$-th factor.  Then there exists a group $Q$ and surjective homomorphisms $\psi_0 \colon G_0 \rightarrow Q$, $\psi_1 \colon 
G_1 \rightarrow Q$ for which 
\begin{equation} \label{fiberedproductofgroups}
G = \{ (g_0,g_1) \in G_0 \times G_1 : \psi_0(g_0) = \psi_1(g_1) \}.
\end{equation}
\end{lemma}
\begin{proof}
See \cite[Lemma (5.2.1)]{ribet}.
\end{proof}
Letting $\psi$ be an abbreviation for the ordered pair $(\psi_0, \psi_1)$, the group $G$ given by \eqref{fiberedproductofgroups} is called the \emph{fibered product of $G_0$ and $G_1$ over $\psi$}, and is commonly denoted by $G_0 \times_\psi G_1$.  Notice that, for a surjective group homomorphism $f \colon Q \rightarrow Q_1$, if $f \circ \psi$ denotes the ordered pair $(f \circ \psi_0, f \circ \psi_1)$ and $G_0 \times_{f \circ \psi} G_1$ denotes the corresponding fibered product, then one has
\begin{equation} \label{fiberedproductcontainment}
G_0 \times_\psi G_1 \subseteq G_0 \times_{f \circ \psi} G_1.
\end{equation}

\begin{lemma} \label{dissolvingundercommutatorslemma}
Let $G_0$ and $G_1$ be groups, let $\psi_0 \colon G_0 \rightarrow Q$ and $\psi_1 \colon G_1 \rightarrow Q$ be a pair of surjective homomorphisms onto a common quotient group $Q$, and let $H = G_0 \times_\psi G_1$ be the associated fibered product.  If $Q$ is cyclic, then one has the following equality of commutator subgroups:
\[
[ H, H ] = [G_0, G_0] \times [G_1, G_1].
\]
\end{lemma}
\begin{proof}
See \cite[Lemma $1$, p. $174$]{langtrotter} (the hypothesis of this lemma is readily verified when $Q$ is cyclic).
\end{proof}
\noindent \emph{Proof of Theorem \ref{completelisttheorem}.}
As shown in \cite{jones2}, one has
\begin{equation*} 
\text{$E$ is not a Serre curve } \; \Longleftrightarrow \;
\begin{matrix}
\exists \text{ a prime $\ell \geq 5$ with } \rho_{E,\ell}(\gal(\ol{\mbq}/\mbq)) \subsetneq \GL_2(\mbz/\ell\mbz), \text{ or} \\
[ \rho_{E,36}(\gal(\ol{\mbq}/\mbq)), \rho_{E,36}(\gal(\ol{\mbq}/\mbq)) ] \subsetneq [ \GL_2(\mbz/36 \mbz), \GL_2(\mbz/36 \mbz) ].
\end{matrix}
\end{equation*}
For each divisor $d$ of $36$, let 
\begin{equation} \label{defofpid}
\pi_{36,d} \colon \GL_2(\mbz/36\mbz) \longrightarrow \GL_2(\mbz/d\mbz)
\end{equation}
denote the canonical projection.  One checks that, for $\ell \in \{ 2, 3 \}$, any proper subgroup $H \subsetneq \GL_2(\mbz/\ell\mbz)$ for which $\det(H) = (\mbz/\ell\mbz)^\times$ must satisfy $[ H, H ] \subsetneq [ \GL_2(\mbz/\ell\mbz), \GL_2(\mbz/\ell \mbz) ]$.  We then define 
\begin{equation} \label{defofG36}
\mc{G}_{36} := \left\{ H \subseteq \GL_2(\mbz/36\mbz) : \; 
\begin{matrix}
\forall d \in \{2, 3 \}, \, \pi_{36,d}(H) = \GL_2(\mbz/d\mbz), \, \det(H) = (\mbz/36\mbz)^\times, \\
 \text{and } \, [H,H] \subsetneq [\GL_2(\mbz/36\mbz), \GL_2(\mbz/36\mbz)] 
\end{matrix}
\right\},
\end{equation}
and note that
\begin{equation} \label{importedresult}
\text{$E$ is not a Serre curve } \; \Longleftrightarrow \;
\begin{matrix}
\exists \text{ a prime $\ell$ and } H \in \mc{G}_{\ell,\text{max}} \text{ for which } \rho_{E,\ell}(\gal(\ol{\mbq}/\mbq)) \subseteq H,  \\
\text{or } \; \exists H \in \mc{G}_{36} \text{ for which } \rho_{E,36}(\gal(\ol{\mbq}/\mbq)) \subseteq H.
\end{matrix}
\end{equation}
As in the prime level case, we need only consider \emph{maximal} subgroups $H \in \mc{G}_{36}$, and because of \eqref{modularinterpretation}, only up to conjugation by $\GL_2(\mbz/36\mbz)$.  Thus, we put
\[
\mc{G}_{36,\text{max}} := \{ H \in \mc{G}_{36} : \; \nexists H_1 \in \mc{G}_{36} \text{ with } H \subsetneq H_1 \subsetneq \GL_2(\mbz/36\mbz) \},
\]
we let $\ds \mc{R}_{36} \subseteq \mc{G}_{36,\text{max}}$ be a set of representatives of $\mc{G}_{36,\text{max}}$ modulo $\GL_2(\mbz/36\mbz)$-conjugation, and we set
\[
\mc{E}_{36} := \{ X_H : \; H \in \mc{R}_{36} \}.
\]
The equivalence \eqref{importedresult} now becomes (see \eqref{Ep})
\[
\text{$E$ is not a Serre curve } \; \Longleftrightarrow \;
\begin{matrix}
\exists \text{ a prime $\ell$ and } X_H \in \mc{E}_\ell \text{ for which } j(E) \in j(X_H(\mbq)),  \\
\text{or } \; \exists X_H \in \mc{E}_{36} \text{ for which } j(E) \in j(X_H(\mbq)).
\end{matrix}
\]
Thus, Theorem \ref{completelisttheorem} will follow from the next proposition.
\begin{proposition} \label{finalprop}
With the above notation, one may take
\[
\mc{R}_{36} = \{ \pi_{36,4}^{-1}(H_4'), \pi_{36,4}^{-1}(H_4''), \pi_{36,9}^{-1}(H_9'), \pi_{36,6}^{-1}(H_6') \},
\]
where $\pi_{36,d}$ is as in \eqref{defofpid} and the groups $H_4'$, $H_4''$, $H_9'$ and $H_6'$ are given by \eqref{X4andX9} and \eqref{defofH}.
\end{proposition}
\begin{proof}
Let $H \in \mc{G}_{36,\text{max}}$.  If $\pi_{36,4}(H) \neq \GL_2(\mbz/4\mbz)$, then \cite{dokchitser} shows that $\pi_{36,4}(H) \subseteq H_4'$ or $\pi_{36,4}(H) \subseteq H_4''$, up to conjugation in $\GL_2(\mbz/4\mbz)$.  If $\pi_{36,9}(H) \neq \GL_2(\mbz/9\mbz)$, then \cite{elkies} shows that, up to $\GL_2(\mbz/9\mbz)$-conjugation, one has $\pi_{36,9}(H) \subseteq H_9'$.  Thus, we may now assume that $\pi_{36,4}(H) = \GL_2(\mbz/4\mbz)$ and $\pi_{36,9}(H) = \GL_2(\mbz/9\mbz)$.  By Lemma \ref{goursatlemma}, this implies that there exists a group $Q$ and a pair of surjective homomorphisms
\[
\begin{split}
\psi_4 \colon \GL_2(\mbz/4\mbz) &\longrightarrow Q \\
\psi_9 \colon \GL_2(\mbz/9\mbz) &\longrightarrow Q
\end{split}
\]
for which
$
H = \GL_2(\mbz/4\mbz) \times_\psi \GL_2(\mbz/9\mbz).
$
We will now show that in this case, up to $\GL_2(\mbz/36\mbz)$-conjugation, we have
\begin{equation} \label{Hcontainment}
H \subseteq \{ (g_4, g_9) \in \GL_2(\mbz/4\mbz) \times \GL_2(\mbz/9\mbz) : \;  \theta ( g_9 \pmod{3} ) = g_4 \pmod{2} \},
\end{equation}
where $\theta \colon \GL_2(\mbz/3\mbz) \longrightarrow \GL_2(\mbz/2\mbz)$ is the map given in \eqref{defofpsi}, whose graph determines the level $6$ structure defining the modular curve $X'(6)$.  This will finish the proof of Proposition \ref{finalprop}.

Let us make the following definitions:
\[
\begin{matrix}
N_4 := \ker \psi_4 \subseteq \GL_2(\mbz/4\mbz), & N_9 := \ker \psi_9 \subseteq \GL_2(\mbz/9\mbz) \\
N_2 := \pi_{4,2}(N_4) \subseteq \GL_2(\mbz/2\mbz), & N_3 := \pi_{9,3}(N_9) \subseteq \GL_2(\mbz/3\mbz)\\
Q_2 := \GL_2(\mbz/2\mbz)/ N_2, & Q_3 := \GL_2(\mbz/3\mbz) / N_3,
\end{matrix}
\]
where $\pi_{4,2} \colon \GL_2(\mbz/4\mbz) \rightarrow \GL_2(\mbz/2\mbz)$ and $\pi_{9,3}  \colon \GL_2(\mbz/9\mbz) \rightarrow \GL_2(\mbz/3\mbz)$ denote the canonical projections.  We then have the following exact sequences:
\begin{equation} \label{exactsequences}
\begin{split}
1 \longrightarrow N_9 \longrightarrow &\GL_2(\mbz/9\mbz) \longrightarrow Q \longrightarrow 1 \\
1 \longrightarrow N_4 \longrightarrow &\GL_2(\mbz/4\mbz) \longrightarrow Q \longrightarrow 1 \\
1 \longrightarrow N_3 \longrightarrow &\GL_2(\mbz/3\mbz) \longrightarrow Q_3 \longrightarrow 1 \\
1 \longrightarrow N_2 \longrightarrow &\GL_2(\mbz/2\mbz) \longrightarrow Q_2 \longrightarrow 1,
\end{split}
\end{equation}
as well as
\begin{equation} \label{usefulexactsequences}
\begin{split}
1 \longrightarrow K_2 \longrightarrow &Q \longrightarrow Q_2 \longrightarrow 1 \\
1 \longrightarrow K_3 \longrightarrow &Q \longrightarrow Q_3 \longrightarrow 1,
\end{split}
\end{equation}
where for each $\ell \in \{ 2, 3 \}$, the kernel $\ds K_\ell \simeq \frac{\ker \pi_{\ell^2,\ell}}{N_{\ell^2} \cap \ker \pi_{\ell^2,\ell}} \subseteq \frac{\GL_2(\mbz/\ell^2\mbz)}{N_{\ell^2}} \simeq Q$ is evidently abelian (since $\ker \pi_{\ell^2, \ell}$ is), and has order dividing $\ell^4 = | \ker \pi_{\ell^2,\ell} |$.  We will proceed to prove that 
\begin{equation} \label{Qsurjection}
Q_2 \simeq \GL_2(\mbz/2\mbz) \quad \text{ and } \quad Q_3 \simeq Q,
\end{equation}
which is equivalent to
\[
N_4 \subseteq \ker \pi_{4,2} \quad \text{ and } \quad \ker \pi_{9,3} \subseteq N_9.
\]
Writing $\tilde{\psi}_4 \colon \GL_2(\mbz/4\mbz) \rightarrow Q \rightarrow Q_2 \simeq \GL_2(\mbz/2\mbz)$ and $\tilde{\psi}_9 \colon \GL_2(\mbz/9\mbz) \rightarrow Q \rightarrow Q_2 \simeq \GL_2(\mbz/2\mbz)$, we then see by \eqref{fiberedproductcontainment} that
\[
H = \GL_2(\mbz/4\mbz) \times_\psi \GL_2(\mbz/9\mbz) \subseteq  \GL_2(\mbz/4\mbz) \times_{\tilde{\psi}} \GL_2(\mbz/9\mbz).
\]
Furthermore, it follows from $Q \simeq Q_3$ that $\tilde{\psi}_9$ factors through the projection $\GL_2(\mbz/9\mbz) \rightarrow \GL_2(\mbz/3\mbz)$.  This, together with the uniqueness of $\mc{N}$ in \eqref{exactsequenceforN} and the fact that every automorphism of $\GL_2(\mbz/2\mbz)$ is inner, implies that \eqref{Hcontainment} holds, up to $\GL_2(\mbz/36\mbz)$-conjugation.  Thus, the proof of Proposition \ref{finalprop} is reduced to showing that \eqref{Qsurjection} holds.

We will first show that $Q_2 \simeq \GL_2(\mbz/2\mbz)$.  Suppose on the contrary that $Q_2 \subsetneq \GL_2(\mbz/2\mbz)$.  Looking at the first exact sequence in \eqref{usefulexactsequences}, we see that $Q$ must then be a $2$-group, and since the $K_3$ has order a power of $3$ (possibly $1$), we see that $Q \simeq Q_3$, and the third exact sequence in \eqref{exactsequences} becomes
\[
1 \longrightarrow N_3 \longrightarrow \GL_2(\mbz/3\mbz) \longrightarrow Q \longrightarrow 1.
\] 
The kernel $N_3$ must contain an element $\gs$ of order $3$, and by considering $\GL_2(\mbz/3\mbz)$-conjugates of $\gs$, we find that $| N_3 | \geq 8$.   Since $3$ also divides $| N_3 |$, we see that $ |N_3| \geq 12$, and so $Q$ must be abelian, having order at most $4$.  Furthermore, since $[ \GL_2(\mbz/3\mbz) , \GL_2(\mbz/3\mbz) ] = \SL_2(\mbz/3\mbz)$, we find that $Q$ has order at most $2$, and thus is cyclic.  Applying Lemma \ref{dissolvingundercommutatorslemma}, we find that $[H, H] = [ \GL_2(\mbz/36\mbz) , \GL_2(\mbz/36\mbz) ]$, contradicting \eqref{defofG36}.  Thus, we must have that $Q_2 \simeq \GL_2(\mbz/2\mbz)$.

We will now show that $Q_3 \simeq Q$.  To do this, we will first take a more detailed look at the structure of the group $\GL_2(\mbz/4\mbz)$.  Note the embedding of groups $\GL_2(\mbz/2\mbz) \hookrightarrow \GL_2(\mbz)$ given by
\[
\begin{matrix}
\begin{pmatrix}
1 & 0 \\
0 & 1
\end{pmatrix}
\mapsto 
\begin{pmatrix}
1 & 0 \\
0 & 1
\end{pmatrix}, &
\begin{pmatrix}
1 & 1 \\
1 & 0
\end{pmatrix}
\mapsto 
\begin{pmatrix}
-1 & -1 \\
1 & 0
\end{pmatrix}, &
\begin{pmatrix}
0 & 1 \\
1 & 1
\end{pmatrix}
\mapsto 
\begin{pmatrix}
0 & 1 \\
-1 & -1
\end{pmatrix},  \\
\begin{pmatrix}
0 & 1 \\
1 & 0
\end{pmatrix}
\mapsto 
\begin{pmatrix}
0 & 1 \\
1 & 0
\end{pmatrix}, &
\begin{pmatrix}
1 & 1 \\
0 & 1
\end{pmatrix}
\mapsto 
\begin{pmatrix}
-1 & -1 \\
0 & 1
\end{pmatrix}, &
\begin{pmatrix}
1 & 0 \\
1 & 1
\end{pmatrix}
\mapsto 
\begin{pmatrix}
1 & 0 \\
-1 & -1
\end{pmatrix}.
\end{matrix}
\]
This embedding, followed by reduction modulo $4$, splits the exact sequence 
\[
1 \rightarrow \ker \pi_{4,2} \rightarrow \GL_2(\mbz/4\mbz) \rightarrow \GL_2(\mbz/2\mbz) \rightarrow 1. 
\]
Also note the isomorphism $(\ker \pi_{4,2}, \cdot)  \rightarrow (M_{2\times 2}(\mbz/2\mbz), +)$ given by $I + 2A \mapsto A \pmod{2}$.  These two observations realize $\GL_2(\mbz/4\mbz)$ as a semi-direct product
\begin{equation} \label{semidirect}
\GL_2(\mbz/4\mbz) \simeq \GL_2(\mbz/2\mbz) \ltimes M_{2\times 2}(\mbz/2\mbz),
\end{equation}
where the right-hand factor is an additive group and the action of $\GL_2(\mbz/2\mbz)$ on $M_{2\times 2}(\mbz/2\mbz)$ is by conjugation.  Since $Q_2 \simeq \GL_2(\mbz/2\mbz)$, we see that, under \eqref{semidirect}, one has
\[
N_4 \subseteq M_{2\times 2} ( \mbz/2 \mbz),
\]
and since it is a normal subgroup of $\GL_2(\mbz/4\mbz)$, we see that $N_4$ must be a $\mbz/2\mbz$-subspace which is invariant under $\GL_2(\mbz/2\mbz)$-conjugation.  This implies that one of the equalities in the following table must hold.
\[
\begin{array}{|c|c|} \hline N_4 & Q \\ 
\hline\hline M_{2 \times 2}(\mbz/2\mbz) & \GL_2(\mbz/2\mbz) \\ 
\hline \left\{ A \in M_{2\times 2}(\mbz/2\mbz) : \; \tr A = 0 \right\} & \GL_2(\mbz/2\mbz) \times \{ \pm 1 \}  \\ 
\hline \left\{ 
\begin{pmatrix}
0 & 0 \\
0 & 0
\end{pmatrix}, 
\begin{pmatrix}
1 & 0 \\
0 & 1
\end{pmatrix}, 
\begin{pmatrix}
1 & 1 \\
1 & 0
\end{pmatrix}, 
\begin{pmatrix}
0 & 1 \\
1 & 1
\end{pmatrix}
\right\} & \GL_2(\mbz/2\mbz) \ltimes (\mbz/2\mbz)^2 \\
\hline  \left\{ 
\begin{pmatrix}
0 & 0 \\
0 & 0
\end{pmatrix}, 
\begin{pmatrix}
1 & 1 \\
0 & 1
\end{pmatrix}, 
\begin{pmatrix}
1 & 0 \\
1 & 1
\end{pmatrix}, 
\begin{pmatrix}
0 & 1 \\
1 & 0
\end{pmatrix}
\right\} & \GL_2(\mbz/2\mbz) \ltimes (\mbz/2\mbz)^2 \\
\hline  \left\{ 
\begin{pmatrix}
0 & 0 \\
0 & 0
\end{pmatrix}, 
\begin{pmatrix}
1 & 0 \\
0 & 1
\end{pmatrix} 
\right\} & \PGL_2(\mbz/4\mbz) \\
\hline
\end{array}
\]
(We have omitted from the table the case that $N_4$ is trivial, since then $Q \simeq \GL_2(\mbz/4\mbz)$, which has order $2^5 \cdot 3$ and thus cannot be a quotient of $\GL_2(\mbz/9\mbz)$.)
In the third row of the table, the action of $\GL_2(\mbz/2\mbz)$ on $(\mbz/2\mbz)^2$ defining the semi-direct product is the usual action by matrix multiplication on column vectors, while in the fourth row of the table, the action is defined via
\[
g \cdot
\begin{pmatrix}
x \\ y 
\end{pmatrix}
= 
\begin{cases}
\begin{pmatrix}
x \\ y 
\end{pmatrix} & \text{ if } g \in \left\{ 
\begin{pmatrix}
1 & 0 \\
0 & 1
\end{pmatrix}, 
\begin{pmatrix}
1 & 1 \\
1 & 0
\end{pmatrix}, 
\begin{pmatrix}
0 & 1 \\
1 & 1
\end{pmatrix}
\right\}, \\
\begin{pmatrix}
y \\ x 
\end{pmatrix} & \text{ if } g \in \left\{ 
\begin{pmatrix}
1 & 1 \\
0 & 1
\end{pmatrix}, 
\begin{pmatrix}
1 & 0 \\
1 & 1
\end{pmatrix}, 
\begin{pmatrix}
0 & 1 \\
1 & 0
\end{pmatrix}
\right\}.
\end{cases}
\]
Since $9$ does not divide $|Q|$, the degree of the projection $Q \twoheadrightarrow Q_3$ is either $1$ or $3$.  Inspecting the table above, we see that in all cases except $Q = \GL_2(\mbz/2\mbz)$, either $Q$ has no normal subgroup of order $3$, or for each normal subgroup $K_3 \unlhd Q$ of order $3$, $Q_3 \simeq Q/K_3$ has $\mbz/2\mbz \times \mbz/2\mbz$ as a quotient group.  Since $[ \GL_2(\mbz/3\mbz), \GL_2(\mbz/3\mbz) ] = \SL_2(\mbz/3\mbz)$, the group $\GL_2(\mbz/3\mbz)$ cannot have $\mbz/2\mbz \times \mbz/2\mbz$ as a quotient group, and so we must have $Q \simeq Q_3$ in these cases, as desired.

When $Q = \GL_2(\mbz/2\mbz)$, we must proceed differently.  Suppose that $Q = \GL_2(\mbz/2\mbz)$ and (for the sake of contradiction) that $Q \neq Q_3$, so that the projection $Q \twoheadrightarrow Q_3$ has degree $3$.  Then $Q_3 \simeq \mbz/2\mbz$, which implies that $N_3 = \SL_2(\mbz/3\mbz)$, so that
\[
N_9 \subseteq \pi_{9,3}^{-1} ( \SL_2(\mbz/3\mbz) ) \subseteq \GL_2(\mbz/9\mbz).
\]
Furthermore, the quotient group $\pi_{9,3}^{-1} ( \SL_2(\mbz/3\mbz) ) / N_9 \simeq \mbz/3\mbz$, and in particular is abelian.  A commutator calculation shows that
\[
[ \pi_{9,3}^{-1} ( \SL_2(\mbz/3\mbz) ) , \pi_{9,3}^{-1} ( \SL_2(\mbz/3\mbz) ) ] = \pi_{9,3}^{-1}( \mc{N} ) \cap \SL_2(\mbz/9\mbz),
\]
(see \eqref{defofN}) and that the corresponding quotient group satisfies 
\[
\pi_{9,3}^{-1} ( \SL_2(\mbz/3\mbz) ) / [ \pi_{9,3}^{-1} ( \SL_2(\mbz/3\mbz) ) , \pi_{9,3}^{-1} ( \SL_2(\mbz/3\mbz) ) ] \simeq \mbz/3\mbz \times \mbz/ 3 \mbz.  
\]
Furthermore, fixing a pair of isomorphisms 
\[
\begin{split}
&\eta_1 \colon \left( \left\{ 
\begin{pmatrix}
1 & 0 \\
0 & 1
\end{pmatrix}, 
\begin{pmatrix}
1 & 1 \\
1 & 0
\end{pmatrix}, 
\begin{pmatrix}
0 & 1 \\
1 & 1
\end{pmatrix}
\right\}, \cdot \right)
\longrightarrow \left( \mbz/3\mbz, + \right), \\
&\eta_2 \colon (1 + 3 \cdot \mbz/9\mbz, \cdot ) \longrightarrow (\mbz/3\mbz, +),
\end{split}
\]
and defining the characters
\[
\begin{split}
\chi_1 \colon \pi_{9,3}^{-1} ( \SL_2(\mbz/3\mbz) ) &\longrightarrow \mbz/3\mbz, \\
\chi_2 \colon \pi_{9,3}^{-1} ( \SL_2(\mbz/3\mbz) ) &\longrightarrow \mbz/3\mbz
\end{split}
\]
by $\chi_1 = \eta_1 \circ \theta \circ \pi_{9,3}$ and $\chi_2 = \eta_2 \circ \det$, we have that every homomorphism $\chi \colon \pi_{9,3}^{-1} ( \SL_2(\mbz/3\mbz) ) \rightarrow \mbz/3\mbz$ must satisfy
\[
\chi = a_1 \chi_1 + a_2 \chi_2,
\]
for appropriately chosen $a_1, a_2 \in \mbz/3\mbz$.  In particular, 
\begin{equation} \label{N9askernel}
N_9 = \ker ( a_1 \chi_1 + a_2 \chi_2 )
\end{equation}
for some choice of $a_1, a_2 \in \mbz/3\mbz$.  One checks that 
\[
\exists g \in \GL_2(\mbz/9\mbz), \; x \in \pi_{9,3}^{-1}(\SL_2(\mbz/3\mbz)) \; \text{ for which } \; \chi_1(g x g^{-1}) \neq \chi_1(x),
\]
whereas $\chi_2(g x g^{-1}) = \chi_2(x)$ for any such choice of $g$ and $x$.  Since $N_9$ is a normal subgroup of $\GL_2(\mbz/9\mbz)$, it follows that $a_1 = 0, a_2 \neq 0$ in \eqref{N9askernel}.  This implies that $N_9 = \SL_2(\mbz/9\mbz)$, which contradicts the fact that $\GL_2(\mbz/9\mbz) / N_9 \simeq Q \simeq \GL_2(\mbz/2\mbz)$ is non-abelian.  This contradiction shows that we must have $Q \simeq Q_3$, and this verifies \eqref{Qsurjection}, completing the proof of Proposition \ref{finalprop}.
\end{proof}

As already observed, the proof of Proposition \ref{finalprop} completes the proof of Theorem \ref{completelisttheorem}. \hfill $\Box$

\end{document}